\documentclass[10pt]{amsart}

\usepackage{amsmath, amscd, amssymb}
\usepackage[frame,cmtip,arrow,matrix,line,graph,curve]{xy}
\usepackage{graphpap, color}
\usepackage[mathscr]{eucal}
\usepackage{cancel}
\usepackage{verbatim}

\numberwithin{equation}{section}

\newcommand{\CC}{\mathbb{C}}

\newcommand{\RR}{\mathbb{R}}
\newcommand{\ZZ}{\mathbb{Z}}

\newcommand{\bm}{\mathbf{m}}

\newcommand{\bp}{\mathbf{p}}

\newcommand{\cal}{\mathcal}

\def\cC{{\cal C}}

\def\cE{{\cal E}}
\def\cF{{\cal F}}

\def\cO{{\cal O}}

\def\cV{{\cal V}}

\def\cY{{\cal Y}}

\def\cX{\mathcal{X} }

\def\and{\quad{\rm and}\quad}

\DeclareMathOperator{\id}{id}

\DeclareMathOperator{\rank}{rank}

\newtheorem{prop}{Proposition}[section]
\newtheorem{theo}[prop]{Theorem}
\newtheorem{lemm}[prop]{Lemma}
\newtheorem{coro}[prop]{Corollary}
\newtheorem{rema}[prop]{Remark}

\newtheorem{ques}[prop]{Question}
\newtheorem{defi}[prop]{Definition}

\def\beq{\begin{equation}}
\def\eeq{\end{equation}}

\def\CC{\mathbb{C}}

\def\cO{\mathcal{O}}

\def\Cpx{\mathrm{Cpx}}

\def\Exp{\mathrm{Exp}}
\def\Log{\mathrm{Log}}
\def\log{\mathrm{log}}

\def\an{\mathrm{an}}

\def\GL{\mathrm{GL}}

\title[Vector bundles on elliptic surfaces and logarithmic transforms]{Vector bundles on elliptic surfaces \\ and logarithmic transformations}

\author{Ludmil Katzarkov}

\address{
University of Miami, Department of Mathematics, PO Box 249085, Coral Gables, FL 33124, USA; Institute of Mathematics and Informatics, Bulgarian Academy of Sciences, Acad. G. Bonchev Str. bl. 8, 1113, Sofia, Bulgaria; and National Research University Higher School of Economics, Laboratory for Mirror Symmetry, 6 Usacheva str., Moscow 119048, Russia}

\email{l.katzarkov@miami.edu}

\author{Kyoung-Seog Lee}

\address{Institute of the Mathematical Sciences of the Americas, University of Miami, Department of Mathematics, 1365 Memorial Drive, Ungar 515, Coral Gables, FL 33146}

\email{kyoungseog02@gmail.com}

\thanks{}

\begin{document}

\begin{abstract}
In this paper, we study how certain vector bundles on an elliptic surface are changed under logarithmic transformations.
\end{abstract}

\maketitle

\section{Introduction}

Logarithmic transformation is an operation on an elliptic surface introduced by Kodaira in 1960s. It has been intensively used in many branches of mathematics, e.g. algebraic geometry, gauge theory, topology to name a few. It is well-known that an elliptic surface with multiple fibers can be obtained by performing logarithmic transformations on an elliptic surface without multiple fibers. Logarithmic transformation is an analytic operation and sometimes it makes an algebraic surface into a non-algebraic one. Therefore, it is hard to follow how algebro-geometric information changes under logarithmic transformation in general. \\ 

In this paper, we study how certain vector bundles on elliptic surfaces are changed under logarithmic transformation. To be more precise, let $S \to C$ be an elliptic surface and let $S(m_1,\cdots,m_k)$ be the elliptic surface obtained by performing logarithmic transformations of order $m_1, \cdots, m_k$ along the smooth fibers over $p_1, \cdots, p_k \in C.$ These data induce natural parabolic structures on $C.$ Let $\cV^{\mathrm{vBun}}_{S(m_1, \cdots, m_k)}$ (resp. $\cV^{\mathrm{ss-vBun}}_{S(m_1, \cdots, m_k)}$) be the category of (resp. $H$-semistable) vertical bundles $\cE$ on $S(m_1, \cdots, m_k)$ such that $c_2(\cE)=0$ where $H$ is a certain ample line bundle giving a good metric (cf. \cite{Bauer}). On the base curve, let $\cV^{\mathrm{ParBun}}_{(C,\bp,\bm)}$ (resp. $\cV^{\mathrm{ss-ParBun}}_{(C,\bp,\bm)}$) be the category of (resp. semistable) 
parabolic bundles on $C$ with parabolic structures on the point $p_i$ with weight at $p_i$ belongs to $\frac{1}{m_i} \ZZ \cap [0,1).$ Bauer studied $\cV^{\mathrm{ss-vBun}}_{S(m_1, \cdots, m_k)}$ in terms of $\cV^{\mathrm{ss-ParBun}}_{(C,\bp,\bm)}$ in \cite{Bauer}. Varma generalized Bauer's work to the semistable vertical Higgs bundles with $c_2=0$ in \cite{Varma}. Their works imply the following.

\begin{theo}[Bauer, Varma]
Let $C$ be a smooth projective curve with $g \geq 2$ and $\pi : S \to C$ be an elliptic surface. Then we have the following equivalence.
$$ \cV^{\mathrm{ss-vBun}}_{S(m_1, \cdots, m_k)} \simeq \cV^{\mathrm{ss-ParBun}}_{(C,\bp,\bm)} $$
We have a similar equivalence for semistable vertical Higgs bundles on $S(m_1, \cdots, m_k)$ and semistable parabolic Higgs bundles on $(C, \bp, \bm).$

\end{theo}

On the other hand, Zube constructed mappings between certain class of vector bundles on an elliptic surface and its logarithmic transformations in \cite{Zube}. To be more precise, let $\pi : S \to C$ be an elliptic surface and let $S(m)$ be the elliptic surface obtained by performing a logarithmic transformation of order $m$ along a smooth fiber over $p \in C.$ When we restrict his mapping to $\cV^{\mathrm{vBun}}_{S}$ and $\cV^{\mathrm{vBun}}_{S(m)},$ his result can be stated as follows.

\begin{theo}[Zube]
There are mappings $\Exp : \cV^{\mathrm{vBun}}_{S(m)} \to \cV^{\mathrm{vBun}}_{S}$ and $\Log : \cV^{\mathrm{vBun}}_{S} \to \cV^{\mathrm{vBun}}_{S(m)}.$
\end{theo}

Zube's constructions are defined using complex analytic geometry. Then we have the following natural questions.

\begin{ques}
(1) Can we understand the above mappings in terms of algebraic geometry? \\
(2) Can we describe the above mappings in terms of the base curve $C?$ 
\end{ques}

We study how $\Exp$ and $\Log$ can be understood in terms of base curve $C.$ We show that there is the following commutative diagram.

\begin{theo}
There is the following commutative diagram where the bottom arrow is forgetful map.
\begin{displaymath}
\xymatrix{ 
\cV^{\mathrm{vBun}}_{S(m)} ~~ \ar[r]^{\Exp} & \cV^{\mathrm{vBun}}_{S} \\
\cV^{\mathrm{ParBun}}_{(C,\bp,\bm)} ~~~ \ar[r]_{\mathrm{forgetful}} \ar[u]^{\pi_P^*} & \cV^{\mathrm{Bun}}_{C} \ar[u]_{\pi_P^*} }
\end{displaymath}
\end{theo}

Biswas studied parabolic bundles on $(C, \bp, \bm)$ in terms of orbifold $\cC.$ Let $c : \cC \to C$ be the coarse moduli map (cf \cite{Olsson}). He proved that the category of parabolic bundles on $(C, \bp, \bm)$ is equivalent to the category of bundles on $\cC$ in \cite{Biswas}. Then, we have the following result.

\begin{theo}
There is the following commutative diagram where $c : \cC \to C$ is the coarse moduli map.
\begin{displaymath}
\xymatrix{ 
\cV^{\mathrm{vBun}}_{S} \ar[r]^{\Log} & \cV^{\mathrm{vBun}}_{S(m)} \\
\cV^{\mathrm{Bun}}_{C} \ar[r]_{c^*} \ar[u]^{\pi_P^*} & \cV^{\mathrm{ParBun}}_{(C,\bp,\bm)} \ar[u]_{\pi_P^*} }
\end{displaymath}
\end{theo}

\begin{rema}
Our proofs show that the above results are true when we perform (inverse) logarithmic transformation on an elliptic surface $\pi : S(m_1, \cdots, m_k) \to C$ which has several multiple fibers.
\end{rema}

\subsection*{Conventions}

We will work over $\CC$. In this paper, $C$ will denote a smooth projective curve of genus $g \geq 2.$ We will assume that all multiple fibers of an elliptic surface $\pi : S \to C$ are multiples of smooth elliptic curves. And we will perform logarithmic transformation only along smooth fibers of the elliptic fibration $\pi : S \to C.$ We will assume that $S$ and its logarithmic transformation $S(m)$ are both smooth projective surfaces. In this paper, $H$ will denote an ample divisor on an elliptic surface which induces a good metric in the sense of \cite{Bauer}.

\subsection*{Acknowledgements}

The second named author thanks Han-Bom Moon and Alexander Petkov for helpful discussions about related topics.

\bigskip

\section{Vector bundles on complex projective varieties}

In this section we review some preliminaries about (parabolic) vector bundles on complex projective varieties.

\subsection{GAGA}

In this paper, we often need to do some analytic operations on coherent analytic sheaves and obtain new coherent analytic sheaves. Then a natural question is whether the new coherent analytic sheaves are algebraic. Serre proved the following classical result in \cite{Serre}.

\begin{theo}[Serre, GAGA principle]
Let $X$ be a complex projective variety and $\widetilde{\cF}$ be a coherent analytic sheaf on $X^{\an}.$ Then there is a unique coherent algebraic sheaf $\cF$ on $X$ such that $\cF^{\an} \cong \widetilde{\cF}.$
\end{theo}

\subsection{Parabolic bundles on a smooth projective variety}

We review basic notions and facts about parabolic bundles on a smooth projective variety. See \cite{Biswas} for more details. Let $Y$ be a smooth projective variety and $D$ be a simple normal crossing divisor on $Y.$ Let $D=\sum^h_{\lambda=1} D_{\lambda}$ where $D_{\lambda}$ is an irreducible component of $D.$ 

\begin{defi}
Let $Y$ be a smooth projective variety and $D$ be a divisor. Let $V$ be a torsion-free coherent sheaf on $Y.$ \\
(1) A quasi-parabolic structure on $V$ with respect to $D$ is
a filtration given by the following subsheaves.
$$ V = F_1(V) \supset F_2(V) \supset \cdots \supset F_{l}(V) \supset F_{{l+1}}(V) = V(-D) $$ 
 (2) A parabolic weight is a collection of real numbers $0 \leq a_1 < \cdots < a_l < 1.$ \\
 (3) A parabolic structure on $V$ is a quasi-parabolic structure on $V$ with parabolic weights. \\

The sheaf $V$ with a parabolic structure is called a parabolic sheaf and will be denoted by $V_*.$ When $V$ is a vector bundle, $V_*$ will be called a parabolic vector bundle.
\end{defi}

Let $V_*$ be a parabolic sheaf and let $t \in \RR.$ We can define $V_t$ as 
$$ V_t = F_i(V)(-[t]D) $$
where $a_{i-1} < t - [t] \leq a_i$

\begin{defi}
A parabolic sheaf $W_*$ is a parabolic subsheaf of a parabolic sheaf $V_*$ if \\
(1) $W$ is a subsheaf of $V$ and $V/W$ is torsion-free. \\
(2) $W_t \subset V_t$ for all $t \in \RR.$ \\
(3) For $s < t,$ if $W_s \subset V_t,$ then $W_s=W_t.$
\end{defi}

\begin{defi}
(1) The parabolic degree of a parabolic bundle $V_*$ of rank $r$ is defined as follows.
$$ \mathrm{pardeg} ~ V_* := \int^1_0 \mathrm{deg}(V_t) dt + r \cdot \mathrm{deg} D $$
(2) The parabolic slope of $V_*$ is defined as follows.
$$ \mathrm{par} \mu(V_*) = \frac{\mathrm{pardeg} ~ V_*}{r} $$
\end{defi}

\begin{defi}
A parabolic bundle $V_*$ is stable (resp. semistable) if for any parabolic subbundle $W_*,$ we have $\mathrm{par} \mu(W_*) < \mathrm{par} \mu(V_*)$ (resp. $\mathrm{par} \mu(W_*) \leq \mathrm{par} \mu(V_*)$).
\end{defi}

Let $V_*$ be a parabolic vector bundle on $Y.$ From the definition, we have the following inclusions
$$ V(-D) \subset F_i(V) \subset V $$
which induces the following sequence.
$$ 0 \to F_i(V)/V(-D) \to V/V(-D) \to V/F_i(V) \to 0 $$

Then we have the following filtration on $V|_D$
$$ V|_D = F_1(V|_D) \supset F_2(V|_D) \supset \cdots \supset F_{l}(V|_D) \supset F_{{l+1}}(V|_D) = 0 $$ 
where $F_i(V|_D) = F_i(V)/V(-D).$ \\

Sometimes we use the above filtration on $V|_D$ to give a parabolic structure on $Y.$

\subsection{Orbifold bundles} 

Biswas studied the relation between parabolic bundles and orbifold bundles in \cite{Biswas}. Let $X$ be a smooth projective variety with finite group $\Gamma$-action. 

\begin{defi}
An orbifold sheaf on $X$ is a torsion-free coherent sheaf $V$ on $X$ where the action of $\Gamma$ on $X$ lifts to $V.$
\end{defi}

\begin{defi}
A subsheaf $W$ of an orbifold sheaf $V$ is $\Gamma$-saturated if it is stable under the $\Gamma$-action and $V/W$ is torsion free. An orbifold sheaf $V$ is orbifold stable (resp. semistable) if, for any $\Gamma$-saturated subsheaf $W$ of $V$ with $0 < \rank W < \rank V,$ we have the inequality $\mu(W) < \mu(V)$ (resp. $\mu(W) \leq \mu(V)$). 
\end{defi}

Let $q : X \to Y := X / \Gamma$ be the quotient map and let us assume that $Y$ is smooth. Let $\widetilde{D} \subset X$ be the reduced ramification divisor and let $D=q(\widetilde{D})$ be the parabolic divisor. 

\begin{theo}\cite{Biswas}
The category of parabolic bundles on $Y$ with respect to $D$ satisfying \cite[Assumption 3.2]{Biswas} is equivalent to the category of $\Gamma$ orbifold bundles on $X.$
\end{theo}

\bigskip

\section{Deligne-Mumford analytic stacks} 

In this section, we review some the theory of Deligne-Mumford analytic stacks.

\subsection{Analytic stacks}

We will review some results about analytic stacks following \cite{BN}. Let $\Cpx$ be the category of complex manifolds with the usual Grothendieck topology where coverings are given by topological open covers. There is a 2-category of stacks over $\Cpx.$ We will recall some basic definitions in \cite{BN}.

\begin{defi}
(1) A stack $\cX$ over $\Cpx$ is a pre-Deligne-Mumford analytic stack if there is an epimorphism $u : U \to \cX$ from a complex manifold $U$ and $u$ is representable by local homeomorphisms. \\
(2) A morphism $f : \cX \to \cY$ of pre-Deligne-Mumford analytic stacks is representable if for any map $Y \to \cY$ from a complex manifold $Y$ which is representable by local homeomorphisms, $X = Y \times_{\cY} \cX$ is a complex manifold.
\end{defi}

\begin{defi}\cite[Definition 3.3]{BN}
A pre-Deligne-Mumford analytic stack $\cX$ is a Deligne-Mumford analytic stack if the diagonal $\cX \to \cX \times \cX$ is representable by closed maps with finite fibers.
\end{defi}

\subsection{GAGA for analytic stacks}

We need to work over certain Deligne-Mumford stack in this paper. Therefore we need GAGA theorem for analytic stacks.

\begin{defi}
A stack $\cY$ is a geometric stack is if it is quasi-compact and the diagonal morphism $\cY \to \cY \times \cY$ is affine.
\end{defi}

In \cite{Lurie}, Lurie proved the following result.

\begin{theo}
Let $\cX$ be a proper Deligne-Mumford stack and $\cY$ be a geometric stack of finite type. Then the analytification functor is an equivalence.
\end{theo}

As a corollary, we can see that every vector bundle on a proper Deligne-Mumford stack is algebraic.

\begin{coro}
Let $\cX$ be a proper Deligne-Mumford stack and $\widetilde{\cF}$ be a vector bundle on it. Then there is an algebraic vector bundle $\cF$ such that $\cF^{\an} \cong \widetilde{\cF}.$
\end{coro}

\bigskip

\section{Vector bundles on an elliptic surface and \\ parabolic bundles on its base curve}

Bauer studied vector bundles on an elliptic surface and parabolic bundles on its base curve in \cite{Bauer}. Varma extended his result for Higgs bundles when the genus of the base curve is greater than or equal to 2 in \cite{Varma}.

\subsection{Vector bundles on an elliptic surface and parabolic bundles on its base curve}

Let us recall Bauer's construction.

\begin{defi}
An elliptic surface $S$ is a smooth projective surface with a fibration $\pi : S \to C$ to a smooth projective curve $C$ whose general fiber is a smooth curve of genus one.
\end{defi}

Let $\pi : S \to C$ be an elliptic surface.

\begin{defi}
(1) A divisor $D$ on $S$ is vertical if it is linearly equivalent to $\sum_i a_i F_i$ where $F_i$ are fibers of $\pi.$ \\
(2) A rank $r$ vector bundle $V$ on $S$ is vertical if $V$ has a filtration
$$ 0=V_0 \subset V_1 \subset \cdots \subset V_r=V $$
by sub-bundles $V_i$ with $V_i/V_{i-1} \cong \cO(D_i)$ where $D_i$ are vertical divisors.
\end{defi}

Bauer constructed two functors $\pi_{P}^*$ and $\pi^{P}_*$ and proved the following in \cite{Bauer}.

\begin{prop}\cite{Bauer}
There are two functors 
$$ \pi^{P}_* : \cV^{\mathrm{ss-vBun}}_{S(m_1, \cdots, m_k)} \to \cV^{\mathrm{ss-ParBun}}_{(C,\bp,\bm)} $$
and 
$$ \pi_{P}^* : \cV^{\mathrm{ss-ParBun}}_{(C,\bp,\bm)} \to \cV^{\mathrm{ss-vBun}}_{S(m_1, \cdots, m_k)} $$
which are inverse to each other. Hence $ \cV^{\mathrm{ss-vBun}}_{S(m_1, \cdots, m_k)}$ and $\cV^{\mathrm{ss-ParBun}}_{(C,\bp,\bm)}$ are equivalent.
\end{prop}

Indeed, the functor $\pi_{P}^*$ is defined for any parabolic bundle on $C.$

\bigskip

\section{Zube's construction}

In \cite{Zube}, Zube constructed mappings between bundles on an elliptic surface and bundles on an elliptic surface obtained by a logarithmic transformation. In this section we review Zube's constructions.

\subsection{Logarithmic transform}

Let us recall the notion of logarithmic transformation. We will follow explanation in \cite{Zube}. \\

Let $\pi : S \to C$ be an elliptic surface. Let $p \in C$ be a point in $C$ such that $\pi$ is smooth over $p$ and $\Delta = \{ ~ x ~ | ~ |x-p| < \epsilon \} \subset C$ be a small disc containing $p.$ Let $S_{\Delta} = \pi^{-1}(\Delta).$ Then $S_{\Delta}$ is isomorphic to $(\Delta \times \CC) / {\ZZ^{\oplus 2}}$ where $\ZZ^{\oplus 2}$ acts on $(\Delta \times \CC)$ by $(n_1, n_2) \cdot (x, z) = (x, z + n_1 + n_2 \omega(x)).$ Note that the fiber $f_x$ over $x$ is isomorphic to $\CC /(\ZZ \oplus \ZZ \omega(x)).$ The elliptic surface $S$ is gluing of $S_{\Delta}$ and $S^* = \pi^{-1}(C \setminus \{ p \}).$

\begin{displaymath}
\xymatrix{ 
S_{\Delta} \ar[r] \ar[d] & S \ar[d] & S^* \ar[l] \ar[d]  \\
\Delta \ar[r] & C & C \setminus \{ p \} \ar[l]  }
\end{displaymath}

Let $\gamma(x)$ be a point of order $m$ on each fiber $f_x.$ We can choose a basis of $\ZZ \oplus \ZZ \omega(x)$ so that $m \gamma(x) = q \in \ZZ, \gcd(q,m)=1$ for all $x \in \Delta.$ Let $\widetilde{\Delta} = \{ ~ y ~ | ~ |y| < \epsilon^{\frac{1}{m}} \}$ be a disc with parameter $y$ and we define $S(m)_{\widetilde{\Delta}}$ to be $\widetilde{\Delta} \times \CC / (\ZZ^{\oplus 2} \oplus \ZZ/m)$ where $(n_1, n_2, h) \in \ZZ^{\oplus 2} \oplus \ZZ/m$ acts on $\widetilde{\Delta} \times \CC$ as follows.
$$ (n_1, n_2, h) \cdot (y,z) = (\xi^{h}y, z+n_1+n_2 \omega(y^m)+\frac{hq}{m}) $$
Here $\xi = \exp(\frac{2\pi \sqrt{-1}}{m}).$ Then there is an analytic isomorphism $\Lambda : S(m)^*_{\Delta} \to S^*_{\Delta}$ induced by the following formula.
$$ (y,z) \mapsto (y^m, z-\frac{q}{2 \pi \sqrt{-1}} \log(y) ) $$

\begin{displaymath}
\xymatrix{ 
S(m)_{\widetilde{\Delta}} \ar[r] \ar[dd] & S(m)_{\Delta} \\
 & S(m)^*_{\Delta} \ar[u] \ar[d]^{\Lambda} \\
S_{\Delta} & S^*_{\Delta} \ar[l] }
\end{displaymath}

We define the logarithmic transformation of $S$ to be the surface $S(m)$ obtained by gluing $S(m)_{\Delta}$ and $S^*$ via the analytic isomorphism $\Lambda : S(m)^*_{\Delta} \to S^*_{\Delta}$ as follows.

\begin{displaymath}
\xymatrix{ 
S(m)_{\Delta} \ar[r] \ar[d] & S(m) \ar[d] & S^* \ar[l] \ar[d]  \\
\Delta \ar[r] & C & C \setminus \{ p \} \ar[l]  }
\end{displaymath}

As we mentioned, it is difficult to understand how algebro-geometric information change via logarithmic transformation in general because it is an analytic operation.

\bigskip

\subsection{Zube's construction}

Let us recall Zube's construction. Let $q : X \to Y$ be a finite morphism between complex manifolds and $G$ be the covering group. Let $V$ be a bundle on $Y$ such that $q^*V$ is trivial. Let $\psi_V : q^*V \cong X \times \CC^r$ be a trivialization. \\

For $x \in X$ and $g \in G,$ we have the following diagram
\begin{displaymath}
\xymatrix{ 
\{ gx \} \times \CC^r & & q^*V|_{gx} = V_{q(x)} = q^*V|_x \ar[ll]^{\psi_V(gx) ~~~} \ar[rr]_{~~~ \psi_V(x)} & & \{ x \} \times \CC^r
}
\end{displaymath}
and the factor of automorphy $e^V_g(x)$ is defined as follows.
$$ e^V_g(x) = \psi_V(gx) \psi_V^{-1}(x) $$
Then $e_g(x)$ satisfies the following identity
$$ e^V_{g_1g_2}(x) = e^V_{g_1}(g_2 x) e^V_{g_2}(x) $$
for every $g_1, g_2 \in G.$ From $e^V_g(x)$ we can recover $V$ on $Y$ by taking the quotient $\CC^r \times X/\sim.$ \\
 
Let $\varphi : V \to W$ be a homomorphism between bundles on $Y.$ Let $e^{V,W}(x) = \psi^{-1}_W(x) q^*(\varphi) \psi_V(x).$ Then we have
$$ e^W_g(x) e^{V,W}(x) = e^{V,W}(gx) e^V_g(x) $$
for all $g \in G$ and $x \in X.$ Conversely, any function $e^{V,W} : X \to M(\mathrm{rk} ~ V, \mathrm{rk} ~ W)$ satisfies the above condition gives a homomorphism of bundles $V \to W$ on $Y.$

\bigskip

Let $\cE$ be a torsion free sheaf on $S.$ Then we define $Sing(\cE) = \{ x \in S ~ | ~ \cE_x$ is not a free $\cO_{S,x}$-module$\}$. Zube introduced two categories of torsion free sheaves on $\cV(p)$ and $\cV(p_m)$ as follows. \\

Let $\cV(p)$ be the category whose objects are pairs $(\cE, i)$ where $\cE$ is a torsion free sheaf on $S$ where $\cE|_{f_p}$ is a homogeneous bundle on the fiber $f_p$ and $Sing(\cE) \cap f_p$ is empty. Let $\cV(p_m)$ be the category whose objects are pairs $(\cF, j)$ where $\cF$ is a torsion free sheaf on $S(m)$ where $\cF|_{f_p}$ is a homogeneous bundle on the fiber $f_p$ and $Sing(\cF) \cap f_p$ is empty.

\begin{lemm}\cite[Lemma 1.2]{Zube}
Let $\cE \in \cV(p)$ (resp. $\cF \in \cV(p_m)$). Then there is a disc $\Delta$ around the point $p$ such that the followings are true. \\
(1) For any $m, l \in \ZZ,$ there is an analytic function $e_{m,l} : \Delta \to \GL(r, \CC)$ such that the function $a_{m,l}(x,z)=e_{m,l}(x)$ is a factor of automorphy of the bundle $\cE_{\Delta}.$ Moreover, the function $e_{m,l}$ are uniquely determined upto conjugacy of the bundle. \\
(2) Let $\cF \in \cV(p_m).$ Then there exists a bundle $\Exp(\cF_{\Delta})$ on $S_{\Delta}$ and an isomorphism ${\Lambda^{-1}}^*(\cF^*) \cong \Exp(\cF_{\Delta})^*.$ \\
(3) Let $\cE \in \cV(p).$ Then there exists a bundle $\Log(\cE_{\Delta})$ on $S(m)_{\Delta}$ and an isomorphism ${\Lambda^{-1}}^*(\cE^*) \cong \Log(\cE_{\Delta})^*.$ \\
(4) If $\cE_{\Delta}$ is a line bundle satisfying $\cE \in \cV(p),$ then there exists $m$ bundles $\Log(\cE_{\Delta}) = \Log(\cE_{\Delta}) \otimes \cO(if_p)$ on $S(m)_{\Delta}$ for $i=0, \cdots, m-1$ and $m$ isomorphisms ${\Lambda^{-1}}^*(\cE_{\Delta}^*) \cong \Log(\cE_{\Delta})^*.$ \\
(5) $\Exp(\cO(if_p))=\cO_{S_{\Delta}},$ $\Exp(\Log(\cE_{\Delta}))=\cE_{\Delta},$ $\Exp(\cO_{S(m)_{\Delta}})=\cO_{S_{\Delta}},$ $\Log(\cO_{S_{\Delta}})=\cO_{S(m)_{\Delta}},$ $\Log(\cO_{S_{\Delta}})=\cO(if_p).$
\end{lemm}

Let us briefly review the construction. Recall that we have the following diagram.

\begin{displaymath}
\xymatrix{ 
 & S(m)_{\Delta} \ar[r] \ar[d]^{\pi_m|_{\Delta}} & S(m) \ar[d]^{\pi_m} \\
S(m)_{\widetilde{\Delta}} \ar[d]_{\widetilde{\pi}|_{\widetilde{\Delta}}} \ar[rd]^{q_1} \ar[ru]^{q_2} & \Delta \ar[r] & C \\
\widetilde{\Delta} \ar[rd]_{p_1} \ar[ru]_{p_2} & S_{\Delta} \ar[r] \ar[d]^{\pi|_{\Delta}} & S \ar[d]^{\pi} \\
 & \Delta \ar[r] & C }
\end{displaymath}

Zube constructed $\Exp : \cV(p_m) \to \cV(p)$ as follows. Let $\cF$ be a vector bundle on $\cV(p_m)$ having factor of automorphy $v_{n_1, n_2, h}(y,z)$ (when we pull-back the bundle to $\widetilde{\Delta} \times \CC$, then it becomes trivial). From the above lemma, one can see that the factors of automorphy $v$ which is independent of $z.$ Then $q^*_2\cF$ has the factors of automorphy $v_{n_1, n_2}(y,z)$ defined as follows. 
$$ v_{n_1, n_2, h}(y,z) = v_{n_1, n_2}(\xi^h y, z+\frac{hq}{m}) $$
And it is also independent of $z$ (again, when we pull-back the bundle to $\widetilde{\Delta} \times \CC$, then it becomes trivial). Then one can define a new factors of automorphy $e_{n_1, n_2, h}$ as follows (note that the action on $\widetilde{\Delta} \times \CC$ becomes different).
$$ e_{n_1, n_2, h}(y,z) = (\xi^h y, z+n_1+n_2 \omega(y^h) ) $$
and it induces factors of automorphy $e_{n_1, n_2}(x,z).$ Zube defined a vector bundle $\Exp(\cE_{\Delta})$ on $S_{\Delta}$ as the bundle given via the factors of automorphy $e_{n_1, n_2}(x,z).$ 

\bigskip

Now let us recall the construction of $\Log : \cV(p) \to \cV(p_m).$ Let $\cE$ be a vector bundle on $S$ given by the factor of automorphy $e_{n_1, n_2}(x,z)$ which is independent of $z.$ Then one can define a vector bundle on $S(m)$ by the factor of automorphy $v(y,z)=e(y^m, z).$ Zube defined a vector bundle $\Log(\cF)$ on $S(m)_{\Delta}$ as the bundle given via the factors of automorphy $v.$ 

\bigskip

He also defined other functors $\Log_i : \cV(p) \to \cV(p_m)$ in a similar way. See \cite{Zube} for more details about the construction. 

\bigskip

Using these local constructions, Zube proved the following theorem.

\begin{theo}\cite[Theorem 1.5]{Zube}
There are functors $\Log_i : \cV(p) \to \cV(p_m)$ and $\Exp : \cV(p_m) \to \cV(p)$ satisfying the following properties. \\
(1) We have $\Exp \circ \Log_i(\cE) = \cE$ for $\cE \in \cV(p).$ \\
(2) $\Log_i(\cE_1 \otimes \cE_2, j_1 \otimes j_2) = \Log_i(\cE_1, j_1) \otimes \Log(\cE_2, j_2)$ and $\Exp(\cF_1 \otimes \cF_2, i_1 \otimes i_2) = \Exp(\cF_1, i_1) \otimes \Exp(\cF_2, i_2).$ \\
(3) We have $c_2(\cF) = c_2(\Exp(\cF))$ and $c_2(\cE) = c_2(\Log_i(\cE)).$ \\
(4) $c_1(\cE_J) = c_1(\det(\cE)_{\det{J}})$ and $c_1(\cF_I) = c_1(\det(\cF)_{\det{I}}).$ \\
(5) $\Log_i(\cO, \id) = (\cO(i f_p), \id),$ $\Exp(\cO,\id)=(\cO,\id)$ and $\Exp(\cO(f_p),\id)=(\cO,\id).$
\end{theo}

For our purpose, it is enough to consider vector bundles with trivial local automorphisms. Therefore we need the following.

\begin{defi}
Let $\cV_0(p)$ (resp. $\cV_0(p_m)$) be the full subcategory of $\cV(p)$ (resp. $\cV(p_m)$) whose objects are pairs $(\cE,\id)$ (resp. $(\cF,\id)$). From the above theorem, we see that there are functors $\Log : \cV_0(p) \to \cV_0(p_m)$ and $\Exp : \cV_0(p_m) \to \cV_0(p).$
\end{defi}

\begin{rema}
From now on, the functors $\Log$ are $\Exp$ will mean the functors restricted to $\cV_0(p)$ and $\cV_0(p_m).$
\end{rema}

\bigskip

\section{The functors $\Exp$ and $\Log$ in terms of \\ parabolic bundles on the base curve}

In this section, we discuss how to understand the functors $\Exp : \cV^{\mathrm{vBun}}_{S(m)} \to \cV^{\mathrm{vBun}}_S$ and $\Log : \cV^{\mathrm{vBun}}_{S} \to \cV^{\mathrm{vBun}}_{S(m)}$ in terms of parabolic bundles on the base curve. Our starting point was observing certain similarity among constructions in \cite{Bauer, Biswas, Varma, Zube}. Let us recall the following diagram.
\begin{displaymath}
\xymatrix{ 
 & S(m)_{\Delta} \ar[r] \ar[d]^{\pi_m|_{\Delta}} & S(m) \ar[d]^{\pi_m} \\
S(m)_{\widetilde{\Delta}} \ar[d]_{\widetilde{\pi}|_{\widetilde{\Delta}}} \ar[rd]^{q_1} \ar[ru]^{q_2} & \Delta \ar[r] & C \\
\widetilde{\Delta} \ar[rd]_{p_1} \ar[ru]_{p_2} & S_{\Delta} \ar[r] \ar[d]^{\pi|_{\Delta}} & S \ar[d]^{\pi} \\
 & \Delta \ar[r] & C }
\end{displaymath}
Note that we have two $\ZZ/m$-actions on $S(m)_{\widetilde{\Delta}}.$ Let $\Gamma_1$ be the group isomorphic to $\ZZ/m$ acting on $S(m)_{\widetilde{\Delta}}$ whose quotient is $S_{\Delta}.$ Let $\Gamma_2$ be the group isomorphic to $\ZZ/m$ acting on $S(m)_{\widetilde{\Delta}}$ whose quotient is $S(m)_{\Delta}.$ Let $\Gamma$ be the group isomorphic to $\ZZ/m$ acting on $\widetilde{\Delta}$ whose quotient is $\Delta.$

\subsection{The functor $\Exp$ in terms of parabolic bundles on the base curve}

\begin{theo}
There is the following commutative diagram.
\begin{displaymath}
\xymatrix{ 
\cV^{\mathrm{vBun}}_{S(m)} \ar[r]^{\Exp} & \cV^{\mathrm{vBun}}_{S} \\
\cV^{\mathrm{ParBun}}_{(C,p,m)} ~~~ \ar[r]_{\mathrm{forgetful}} \ar[u]^{\pi^*_P} & \cV^{\mathrm{Bun}}_{C} \ar[u]_{\pi^*_P} }
\end{displaymath}

\end{theo}
\begin{proof}
There is a Riemann surface $\widetilde{C}$ with Galois group $\Gamma$ action whose quotient map $\widetilde{C} \to C$ is ramified at $p.$ Then we have the following commutative diagram
\begin{displaymath}
\xymatrix{ 
\widetilde{S(m)} \ar[r] \ar[d]_{\widetilde{\pi}} & S(m) \ar[d]^{\pi_m} \\
\widetilde{C} \ar[r] & C }
\end{displaymath}
such that $\widetilde{S(m)} \to \widetilde{C}$ does not have a multiple fiber and $\widetilde{S(m)} \to S(m)$ is \'etale (cf. \cite{FM, Varma}). Near the ramification point $p \in \Delta,$ the diagram is analytically isomorphic to the following diagram.
\begin{displaymath}
\xymatrix{ 
S(m)_{\widetilde{\Delta}} \ar[d]_{\widetilde{\pi}|_{\widetilde{\Delta}}} \ar[r]^{q_2} & S(m)_{\Delta} \ar[r] \ar[d]^{\pi_m|_{\Delta}} & S(m) \ar[d]^{\pi_m} \\
\widetilde{\Delta} \ar[r]_{q_1} & \Delta \ar[r] & C }
\end{displaymath}

Then the functor $\cV^{\mathrm{ParBun}}_{(C,p,m)} \to \cV^{\mathrm{vBun}}_{S(m)}$ can be understood using the following diagram where the horizontal arrows are equivalences (cf. \cite{Biswas}).
\begin{displaymath}
\xymatrix{ 
 \cV^{\Gamma_2-\mathrm{vBun}}_{\widetilde{S(m)}} \ar[r] & \cV^{\mathrm{vBun}}_{S(m)} \\
\cV^{\Gamma_2-\mathrm{Bun}}_{\widetilde{C}} \ar[r] \ar[u]^{\widetilde{\pi}^*} & \cV^{\mathrm{ParBun}}_{(C,p,m)} \ar[u]_{\pi_P^*} }
\end{displaymath}

Let $F$ be a parabolic bundle on $(C, p,m)$ and let $\cF = \pi_P^* F.$ Then we see that the pullback of $\cF_{\Delta}$ by $q_2$ is isomorphic to $(\widetilde{\pi}|_{\widetilde{\Delta}})^*q_1^*(F_{\Delta})$ which has both $\Gamma_1$ and $\Gamma_2$ actions since $\widetilde{\pi}|_{\widetilde{\Delta}}$ is equivariant with both $\Gamma_1$ and $\Gamma_2$ actions. Because we have the following commutative diagram and 
\begin{displaymath}
\xymatrix{ 
S(m)_{\widetilde{\Delta}} \ar[d]_{\widetilde{\pi}|_{\widetilde{\Delta}}} \ar[r]^{p_2} & S_{\Delta} \ar[d]^{\pi|_{\Delta}} \\
\widetilde{\Delta} \ar[r]^{p_1} & \Delta
}
\end{displaymath}
we have $({\pi|_{\Delta}})_* ({p_2}_*)^{\Gamma_1} \cong ({p_1}_*)^{\Gamma} (\widetilde{\pi}|_{\widetilde{\Delta}})_*$ when it applies to a $\Gamma_1$-equivariant sheaf. It implies that $({\pi|_{\Delta}})_* ({p_2}_*)^{\Gamma_1} (\widetilde{\pi}|_{\widetilde{\Delta}})^*q_1^*(F_{\Delta}) \cong ({p_1}_*)^{\Gamma} q_1^*(F_{\Delta}).$ Note that $p_1=q_1.$ Note that $[\widetilde{C}/\Gamma]$ is a Deligne-Mumford stack and we can also regard it as a analytic Deligne-Mumford stack. From GAGA principle, we see that the operation $({p_1}_*)^{\Gamma} q_1^*(F_{\Delta})$ uniquely determine a vector bundle and this operation corresponds to forgetting parabolic structure at $p.$ Therefore we obtain the desired result.
\end{proof}

\subsection{The functor $\Log$ in terms of parabolic bundles on the base curve}

Now we prove that the functor $\Log$ corresponds to pullback with respect to $c : \cC \to C.$

\begin{theo}
There is the following commutative diagram. 
\begin{displaymath}
\xymatrix{ 
\cV^{\mathrm{vBun}}_{S} \ar[r]^{\Log} & \cV^{\mathrm{vBun}}_{S(m)} \\
\cV^{\mathrm{Bun}}_{C} \ar[r]_{c^*}  \ar[u]^{\pi^*_P} & \cV^{\mathrm{ParBun}}_{(C,\bp,\bm)}  \ar[u]_{\pi^*_P} }
\end{displaymath}

\end{theo}
\begin{proof}
Let $E$ be a vector bundle on $C$ and let $\cE = \pi_P^* E.$ From our previous discussion, we see that the pullback of $\Log(\cE)_{\Delta}$ by $q_2$ is isomorphic to a $\Gamma_2$-equivariant sheaf $p^*_2(\cE_{\Delta})$ and hence isomorphic to $p^*_2 (\pi|_{\Delta})^* (E_{\Delta})$ near the point $p \in \Delta \subset C.$

Let us consider the following Cartesian diagram.
\begin{displaymath}
\xymatrix{ 
S(m)_{\widetilde{\Delta}} \ar[d]_{\widetilde{\pi}|_{\widetilde{\Delta}}} \ar[r]^{p_2} & S_{\Delta} \ar[d]^{\pi|_{\Delta}} \\
\widetilde{\Delta} \ar[r]^{p_1} & \Delta
}
\end{displaymath}

Then we have $p^*_2 (\pi|_{\Delta})^* (E_{\Delta}) \cong ({\widetilde{\pi}|_{\widetilde{\Delta}}})^* p^*_1 (E_{\Delta}).$ Note that this bundle has $\Gamma$-action. From the GAGA principle, we see that considering $({\widetilde{\pi}|_{\widetilde{\Delta}}})^* p^*_1 (E_{\Delta})$ with $\Gamma$-action corresponds to giving parabolic structure at $p.$ Therefore we obtain the desired result for the functor $\Log.$
\end{proof}

\begin{rema}
The same proofs show that the above results are true when we perform (inverse) logarithmic transformation on an elliptic surface $\pi : S(m_1, \cdots, m_k) \to C$ which has several multiple fibers.
\end{rema}

\bigskip

\section{Further directions}

In this section, we discuss several further directions. We also ask some natural questions.

\subsection{Vertical vector bundles with nonzero $c_2$}

Indeed, Zube's constructions work for certain bundles with nonzero $c_2$ and it is an interesting question whether there is an algebro-geometric way to understand the operations $\Exp,$ $\Log$ for those bundles.

\subsection{Elliptic surfaces with singular multiple fibers} 

It is also interesting to generalize Zube's work for the cases when the multiple fibers are singular. We need some extra care to deal with these cases.

\subsection{Gauge theory}

One of our original motivations for this work is to understand some gauge theoretic constructions in terms of algebraic geometry. It will be very nice if we can recover some of the classical results in gauge theory using our description. In order to do it, one need to extend our work to the vertical vector bundles with nonzero $c_2$.

\subsection{Derived category}

One of our original motivations for this work is to understand how derived category of coherent sheaves on an elliptic surface changes under the logarithmic transformation which seems to be a very challenging task at this moment. Recently, derived categories of some elliptic surfaces were studied and it turns out that there are (quasi-)phantom categories in the derived categories of certain Dolgachev surfaces and non-minimal Enriques surfaces (cf. \cite{Cho, CL}). On the other hand, there are several examples where one can find that the derived category of coherent sheaves and moduli space of vector bundles on a variety are closely related (see, for example \cite{LM22} and references therein). Therefore we expect that the change of vector bundles on elliptic surfaces via logarithmic transformation should be related to the change of derived categories. For example, we can perform logarithmic transformations twice to obtain an Enriques surface, and then we have the non-trivial fundamental group $\ZZ/2.$ It is well-known that representations of the fundamental group of a variety are closed related to the vector bundles living on it (cf. \cite{NS}). On the other hand, if there is a quasi-phantom category on the derived category of blown-up of this surface, then the Grothendieck group of the quasi-phantom category is $\ZZ/2.$ It is an interesting question whether these two phenomena are related.

\bigskip

\bibliographystyle{amsplain}

\begin{thebibliography}{99}

\bibitem{Atiyah} M. Atiyah. \emph{Vector bundles over an elliptic curve.} Proc. London Math. Soc. (3) 7 (1957), 414-452.

\bibitem{Bauer} S. Bauer. \emph{Parabolic bundles, elliptic surfaces and SU(2)-representation spaces of genus zero.} Fuchsian groups. Math. Ann. 290 (1991), no. 3, 509-526.

\bibitem{BN} K. Behrend and B. Noohi. \emph{Uniformization of Deligne-Mumford curves.} Journal f\"ur die reine und angewandte Mathematik (Crelles Journal), vol. 2006, no. 599, 2006, 111-153.

\bibitem{Biswas} I. Biswas. \emph{Parabolic bundles as orbifold bundles.} Duke Math. J. 88(2), 1997, 305-325.

\bibitem{Cho} Y. Cho. \emph{Exceptional collections on nonminimal Enriques surfaces.} Proc. Amer. Math. Soc. 150 (2022), 5-14.

\bibitem{CL} Y. Cho and Y. Lee. \emph{Exceptional collections on Dolgachev surfaces associated with degenerations.} Adv. Math. 324 (2018), 394-436.

\bibitem{FM} R. Friedman and J. W. Morgan. Smooth four-manifolds and complex surfaces. Ergebnisse der Mathematik und ihrer Grenzgebiete (3) [Results in Mathematics and Related Areas (3)], 27. Springer-Verlag, Berlin, 1994. x+520 pp.

\bibitem{Hall} J. Hall. \emph{Generalizing the GAGA Principle.} Preprint, arXiv:1101.5123.

\bibitem{HL} D. Huybrechts and M. Lehn. The geometry of moduli spaces of sheaves. Second edition. Cambridge Mathematical Library. Cambridge University Press, Cambridge, 2010. xviii+325 pp.

\bibitem{LM22} K.-S. Lee and H.-B. Moon. \emph{Derived category and ACM bundles of moduli space of vector bundles on a curve.} Preprint, arXiv:2201.10033.

\bibitem{Lurie} J. Lurie. \emph{Tannaka duality for geometric stacks.} Preprint, arXiv:math/0412266.

\bibitem{NS} M. S. Narasimhan and C. S. Seshadri. \emph{Stable and unitary vector bundles on a compact Riemann surface}. Annals of Mathematics, 82(3), 540-567. 

\bibitem{Olsson} M. Olsson. Algebraic spaces and stacks. American Mathematical Society Colloquium Publications, 62. American Mathematical Society, Providence, RI, 2016. xi+298 pp.

\bibitem{Serre} J.-P. Serre. \emph{Géométrie algébrique et géométrie analytique.} Université de Grenoble. Annales de l’Institut Fourier 6: 1-42, (1956).

\bibitem{Varma} R. Varma. \emph{On Higgs bundles on elliptic surfaces.} The Quarterly Journal of Mathematics, Volume 66, Issue 3, September 2015, 991-1008.

\bibitem{Zube} S. Zube. \emph{Logarithmic transformation of an elliptic surface and vector bundles.} (Russian) Izv. Akad. Nauk SSSR Ser. Mat. 55 (1991), no. 3, 466-482; translation in Math. USSR-Izv. 38 (1992), no. 3, 455-469.

\end{thebibliography}

\end{document}